\documentclass[11pt]{article}

\usepackage[a4paper]{geometry}
\usepackage{times}
\usepackage[english]{babel}
\usepackage[latin1]{inputenc}
\usepackage{lmodern}
\usepackage[T1]{fontenc}
\usepackage{amsmath,amssymb}
\usepackage{amsthm}
\usepackage{amsfonts}
\usepackage[pdftex]{color, graphicx}
\usepackage{tikz}
\usepackage{hyperref}

\theoremstyle{plain}
	\newtheorem{theorem}{Theorem}[section]
	\newtheorem{proposition}[theorem]{Proposition}
	\newtheorem{lemma}[theorem]{Lemma}
	\newtheorem{corollary}[theorem]{Corollary}
	\newtheorem*{claim}{Claim}
	
\theoremstyle{definition}
	\newtheorem{definition}[theorem]{Definition}
	\newtheorem{notation}[theorem]{Notation}

	\newtheorem{question}[theorem]{Question}
\theoremstyle{remark}
	\newtheorem{remark}[theorem]{Remark}

\newcommand{\RR}{\mathbb{R}}

\newcommand{\QQ}{\mathbb{Q}}
\newcommand{\NN}{\mathbb{N}}
\DeclareMathOperator{\dom}{dom}
\DeclareMathOperator{\last}{last}
\newcommand{\CQQ}{\alg_\RR(\QQ)}
\DeclareMathOperator{\alg}{alg}
\DeclareMathOperator{\IS}{IS}
\newcommand{\res}{\upharpoonright}
\newcommand{\cp}[1]{\left( #1 \right)}

\newcommand{\ap}[1]{\langle #1 \rangle}
\newcommand{\bp}[1]{\left\lbrace #1 \right\rbrace}

\newcommand{\vgt}[1]{``#1''}

\title{Some algebraic equivalent forms of $\RR \subseteq L$ }
\author{Silvia Steila}
\date{}

\begin{document}
\maketitle

\begin{abstract}
We study $\Sigma^1_2$ definable counterparts for some algebraic equivalent forms of the Continuum Hypothesis. All turn out to be equivalent to \vgt{all reals are constructible}.
\end{abstract}

\section{Introduction}

Sierpinski showed that the Continuum Hypothesis (CH) holds if and only if there are sets $A, B \subseteq \RR^2$ such that $A \cup B =\RR^2$ and for any $a,b \in \RR$ the sections $A_a=\bp{y : (a,y) \in A}$ and $B^b= \bp{x : (x,b) \in B}$ are countable \cite{Sierpinski}. In \cite{TornquistWeiss1, TornquistWeiss} T\"{o}rnquist and Weiss studied many $\Sigma^1_2$ definable versions of some equivalent forms of CH which happen to be equivalent to \vgt{all reals are constructible}. For instance, they proved the $\Sigma^1_2$ counterpart of Sierpinski's equivalence: $\RR \subseteq L$ if and only if there are $\Sigma^1_2$ sets $A, B \subseteq \RR^2$  such that $A_1 \cup A_2 =\RR^2$ and for any $a,b \in \RR$ all sections $A_a=\bp{y : (a,y) \in A}$ and $B^b= \bp{x : (x,b) \in B}$ are countable.

We follow their scheme to get some algebraic forms of $\RR \subseteq L$. While T\"{o}rnquist and Weiss considered $\Sigma^2_1$ statements of the form \vgt{there exist finitely many objects such that something happens}, the first algebraic statements we analyze, namely the $\Sigma^1_2$ counterparts of Erd\H{o}s and Kakutani's equivalence  \cite{ErdosKakutani} and of Zoli's equivalence  \cite{Zoli}, require the existence of countably many objects. 
\begin{theorem}\label{Thm:1}
The following are equivalent:
\begin{enumerate}
\item $\RR \subseteq L$;
\item  there is a countable partition of $\RR$ into $\Sigma^1_2$-uniformly definable subsets consisting only of rationally independent numbers;
\item the set of all transcendental reals is the union of countably many uniformly $\Sigma^1_2$ definable algebraically independent subsets.
\end{enumerate}
\end{theorem}
As side results, the proof we present for the $\Sigma^1_2$ definable versions provides a generalization of both the equivalences by Erd\H{o}s-Kakutani and Zoli, where CH and \vgt{countably many} in the original theorems are replaced by $2^{\aleph_0}\leq \kappa^+$ and \vgt{$\kappa$-many}.

Then we study the $\Sigma^1_2$ version of polynomial avoidance and Schmerl's results \cite{Schmerl}, by introducing a $\Sigma^1_2$ version of $m$-avoidance, for $m \in \omega$. As a corollary we obtain the $\Sigma^1_2$ counterpart of a theorem by Erd\H{o}s and Komj\'{a}th  \cite{ErdosKomjath}: 
\begin{theorem}\label{Thm:2}
$\RR \subseteq L$ if and only if there exists a $\Sigma^1_2$ coloring of the plane in countably many colors with no monochromatic right-angled triangle. 
\end{theorem}

\paragraph{Plan of the paper.} The main results are organized in three sections. In Section \ref{section: ErdosKakutani} we prove the first part of Theorem \ref{Thm:1} (i.e. (1) iff (2)) and in Section \ref{section: Zoli} we prove the second part of Theorem \ref{Thm:1}  (i.e. (1) iff (3)). Schmerl's results and Theorem \ref{Thm:2} are shown in Section \ref{section: Schmerl}. Each section starts with a short introduction and a generalization of the classical result, before presenting the definable counterpart.

\section{Preliminaries}\label{section: preliminaries}
A set is $\Sigma^1_2$ if there exists a $\Sigma^1_2$ predicate which defines it, and a function is $\Sigma^1_2$ if its graph is. A set is $\Delta^1_2$ if both it and its complement are $\Sigma^1_2$. Observe that all notions are intended lightface. For details we refer to \cite{Moschovakis}.

\begin{definition}
A $\Delta^1_2$ well-ordering $\prec$ is \emph{strong} if it has length $\omega_1$ and if for any $P \subseteq \RR \times \RR$ which is $\Sigma^1_2$,
\[
	\forall z \prec y P(x,z)
\]
is $\Sigma^1_2$ as well.
\end{definition}
For short we denote $x_{\prec} = \bp{z: z \prec x}$. Given a $\Delta^1_2$ strong well-ordering $\prec$, $P \subseteq \RR^{<\omega} \times \RR$ $\Sigma^1_2$ and $x, y \in \RR$, $\forall s \in (x_{\prec})^{<\omega} P(s,y)$ is $\Sigma^1_2$.
The existence of a $\Delta^1_2$ strong well-ordering of $\RR$ is equivalent to requiring that the initial segment relation $\IS \subseteq \RR \times \RR^{{\leq}\omega}$ defined by
\[
	\IS(x,y) \iff \forall z (z \prec x \iff  \exists n(y(n)= z)) \wedge \forall i, j (y(i) = y(j) \implies i=j),
\]
is $\Delta^1_2$.  We also use the function $\IS^*: \RR \to \RR^{\leq \omega}$ which defines the initial segment of a given real:
\[
	\IS^*(x)= v \iff \IS(x,v) \wedge (\forall w \prec^* v)\neg\IS(x,w).
\]
where $\prec^*$ is the product order in $\RR^{{\leq}\omega}$ induced by $\prec$. If $\RR \subseteq L$ then there exists a $\Delta^1_2$ strong well-ordering of reals which is the usual well-ordering of $\RR$ in $L$ (see e.g. \cite{kanamori}). \medskip 
 
For short, let $S$ be an equivalent form of CH. As shown by T\"{o}rnquist and Weiss, proving that CH implies $S$ can often be directly made into a proof of the effective implication, using the $\Delta^1_2$ strong well-ordering of reals. 

Vice versa, from a proof of \vgt{$S$ implies CH}  we cannot usually extract a proof of the effective implication. To this end, we need some properties of $L$, mainly a corollary of a theorem by Mansfield and Solovay: if a $\Sigma^1_2$ set contains a non-constructible real then it is uncountable. This result does not explicitly appear in this paper, since the use of the perfect set property is hidden in the proof of the $\Sigma^1_2$ version of two partition results, proved by T\"{o}rnquist and Weiss, that we are going to apply. Both are $\Sigma^1_2$ counterparts of partition results by Komj{\'a}th and Totik \cite{KomjathTotik}.

\begin{proposition}[Komj\'{a}th, Totik]\label{prop: KomjathTotik} \label{cor: KomjathTotik} ~ 
\begin{enumerate}
\item  Let $\kappa$ be an infinite cardinal, $|A|= \kappa^+$, $|B| = (\kappa^+)^+$, and $k \in \NN$. If $f: A \times B \to \kappa$, then there exist $A' \subseteq A$, $B' \subseteq B$, $|A'|=|B'|=k$ such that $A' \times B'$ is monochromatic.
\item  If $\neg CH$, then for any coloring $g: \RR \to \omega$ there are distinct $x_{00},x_{01},x_{10},x_{11} \in \RR$ of the same color such that 
$x_{00}+x_{11}=x_{01}+x_{10}.$
\end{enumerate}
\end{proposition}

\begin{proposition}[T\"{o}rnquist, Weiss]\label{theorem: CD} \label{theorem: schur} ~ 
\begin{enumerate}
\item There is a non-constructible real if and only if for any $x \in \RR\cap L$, for any $\Sigma^1_2(x)$ coloring $f: \RR \times \RR \rightarrow \omega$ and for any $k\in \omega$ there are $C,\ D \subseteq \RR$ such that $|C|=|D|=k$ and $f \res C \times D$ is monochromatic.
\item There is a non-constructible real if and only if for any $\Sigma^1_2$ coloring $g: \RR\rightarrow \omega$ there are four distinct $x_{00},x_{01},x_{10},x_{11}\in \RR$ of the same color such that
$
	x_{00} + x_{11} = x_{01} + x_{10}.
$
\end{enumerate}
\end{proposition}

We mainly work in $\RR$, but sometimes we also work in different recursively presented Polish spaces, as $\RR^{\leq \omega}$. This is not a problem in our setting, since between any two recursively presented Polish spaces there is a $\Delta^1_1$ bijection. 

\section{Rationally independent sets} \label{section: ErdosKakutani}

In \cite{ErdosKakutani} Erd\H{o}s and Kakutani proved there is a close connection between CH and the existence of some special rationally independent subsets of reals. We prove a generalization of Erd\H{o}s and Kakutani's equivalence. Recall that a set $X \subseteq \RR$ is rationally independent if for any $n \in \NN$,  $x_0, \dots, x_{n-1} \in X$ and for any $q_0, \dots, q_{n-1} \in \QQ\setminus\bp{0}$ we have:
\[
	\sum_{i=0}^{n-1} q_ix_i \neq 0.
\]

First of all, following Komj{\'a}th and Totik, we get a straightforward generalization of Proposition \ref{cor: KomjathTotik}.2. 

\begin{proposition}\label{prop: KomjathTotikK}
If $2^{\aleph_0} \geq (\kappa^+)^+$, then for any coloring $g: \RR \to \kappa$ there are distinct $x_{00},x_{01},x_{10},x_{11} \in \RR$ of the same color such that  $x_{00}+x_{11}=x_{01}+x_{10}.$
\end{proposition}
\begin{proof}
Assume that $2^{\aleph_0} \geq (\kappa^+)^+$ and consider any $g:\RR\to \kappa$. Take an injection $i:(\kappa^+)^+ \hookrightarrow \RR$ and define $i':(\kappa^+)^+ \hookrightarrow \RR$ such that $\mbox{ran}(i')$ is rationally independent. Put $i'(\alpha) = i(\beta)$ for $\beta= \mu \gamma \in (\kappa^+)^+ (\forall s \in (\QQ\setminus \bp{0})^{<\omega} \forall t \in \gamma^{<\omega} ( \dom(s) = \dom(t) \implies \sum_{j=0}^{dom(s)} s(j)i(t(j)) \neq i(\gamma)))$.
Let 
\[
	\bp{a_\alpha : \alpha < \kappa^+}\cup \bp{b_\beta : \beta < (\kappa^+)^+}
\]
be a rationally independent set, and define the following coloring:
\[
\begin{array}{rrcl}
	f: & \kappa^+ \times (\kappa^+)^+  & \longrightarrow & \kappa \\
	   & (\alpha, \beta) &\longmapsto & g(a_\alpha + b_\beta)
\end{array}
\]
Thanks to Proposition \ref{prop: KomjathTotik}.1 there exist $\alpha_0, \alpha_1\in \kappa^+$ and $\beta_0, \beta_1 \in (\kappa^+)^+ $ such that $\bp{\alpha_0, \alpha_1} \times \bp{\beta_0, \beta_1}$ is monochromatic. Define $x_{ij}= \alpha_i + \beta_j$ for any $i, j \in 2$. They are distinct and they satisfy $x_{00}+x_{11}=x_{01}+x_{10}$.
\end{proof}
Recall that $H \subseteq \RR$ is a Hamel basis if both $H$ is rationally independent and $H$ is a basis of $\RR$ over $\QQ$.

\begin{theorem}\label{theorem: ErdosKakutaniK}
We have $2^{\aleph_0} \leq \kappa^+$ if and only if $\RR \setminus \bp{0}$ can be covered by $\kappa$-many rationally independent sets.
\end{theorem}
\begin{proof} 
\vgt{$\Rightarrow$}. Assume that $2^{\aleph_0} \leq \kappa^+$, and let $H$ be a Hamel basis for $\RR$. Take an injection $f: \RR \hookrightarrow \kappa^+$ and define the order $\vartriangleleft$ of length ${\leq}\kappa^+$ of $\RR$ by $x \vartriangleleft y$ if and only if $f(x) <f(y)$. For any natural number $n >0$ and for any $s \in (\QQ \setminus \bp{0})^{n}$, define
\begin{equation}\label{form:hamel1a}
	s\cdot H^n= \bp{x: \exists h_0 \vartriangleleft  \dots \vartriangleleft  h_{n-1} \in H \cp{ x= \sum_{i=0}^{n-1} s(i) h_i}}.
\end{equation}

First notice that as $H$ is a Hamel basis, $\RR \setminus \bp{0}$ is covered by all sets $s\cdot H^n$ for $s \in (\QQ \setminus \bp{0})^n$. So it suffices to show the result for all sets $s \cdot H^n$. From now on fix $n \in \NN$ and $s \in (\QQ \setminus \bp{0})^{n}$. Note if $n=1$ as $H$ is rationally independent the result is trivial, so assume that $n>1$.

Given $x \in s\cdot H^n$, define $\last(x)$ as the greatest element of $H$ which appears in (\ref{form:hamel1a}) (i.e. $h_{n-1}$). For any $h \in H$ define
\[
	s \cdot (H^{n-1}h) = \bp{ x \in s \cdot H^n : \last(x)=h}.
\]
Given $h \in H$ let $\gamma_h = |s \cdot (H^{n-1}h)|$. Observe that $\gamma_h \leq \kappa$ since $|H|=\kappa^+$. Therefore for any $h \in H$ we can fix an enumeration:
\[
	s \cdot (H^{n-1}h)=\bp{x^h_\alpha: \alpha < \gamma_h}.
\]
Finally, for any $\alpha < \kappa$ let $S_{s,\alpha}$ be the set of $\alpha$-th elements of any $s \cdot (H^{n-1}h)$ for $h \in H$; i.e.
\[
	S_{s,\alpha}=\bp{x^h_\alpha : h \in H  \wedge \alpha < \gamma_h}.
\]
Therefore we have
\[
	\RR \setminus \bp{0}=\bigcup \bp{S_{s,\alpha}: s \in ( \QQ \setminus \bp{0})^{<\omega} \wedge \alpha <\kappa}.
\]
We claim that for any $\alpha < \kappa$, $S_{s,\alpha}$ is rationally independent. Indeed, assume by contradiction that
\begin{equation}\label{form:hamel2a}
	   		\sum_{i=0}^{k} p_i x_i=0,
\end{equation}
where for any $i \in k+1$:
\begin{itemize}
	   \item  $x_i \in S_{s,\alpha}$;
	   \item  for any $j \in k+1$, $x_i \neq x_j$;
	   \item  $p_i$ is a not null integer.
\end{itemize}  
By construction, for any two distinct elements $x_1, x_2 \in S_{s,\alpha}$, $\last(x_1) \neq \last(x_2)$. Then there would exist an integer $i_0 \in k+1$ for which $\last(x_{i_0}) > \last(x_i) $ for any $i \in k+1 \setminus\bp{i_0}$. Hence in the expansion (\ref{form:hamel1a}) of all $x_i$, $\last(x_{i_0})$ would appear only once, which is impossible because of (\ref{form:hamel2a}) and $H$ rationally independent. 

\vgt{$\Leftarrow$.} Let $\RR\setminus \bp{0} = \bigcup \bp{S_\alpha : \alpha \in \kappa}$, where each $S_\alpha$ is a rationally independent set.  Assume by contradiction that $2^{\aleph_0} > \kappa^+$ and define $g : \RR  \to \kappa$ such that $g(0) = 0$ and for any $x \in \RR \setminus \bp{0}$
\[
	g(x)= \alpha + 1 \iff  x \in S_\alpha \wedge \forall \beta < \alpha (x \notin S_\beta). 
\]
Applying Proposition \ref{prop: KomjathTotikK} we get $x_{00},\ x_{10},\ x_{01},\ x_{11} \in \RR \setminus\bp{0} $ such that $x_{00},\ x_{10},\ x_{01},\ x_{11} \in S_\alpha$ for some $\alpha \in \kappa$ which are rationally dependent. Contradiction.
\end{proof}

The proof of the first implication follows the one by Erd\H{o}s and Kakutani's result, while the argument for the vice versa, as far as we know, is new. The original proof uses a tree argument, which cannot be easily adapted to the $\Sigma^1_2$ version. 

\begin{remark}\label{rmk: disjointHamel}
Note that in the proof of Theorem \ref{theorem: ErdosKakutaniK} the subsets $S_{s,\alpha}$, for $s \in (\QQ\setminus\bp{0})^{<\omega}$ and $\alpha \in \kappa$, are disjoint. 
\end{remark} 

\subsection{Definable counterpart}
We prove that $\RR \subseteq L$ holds if and only if $\RR \setminus \bp{0}$ can be decomposed in countably many rationally independent subsets of reals which are uniformly definable by a $\Sigma^1_2$ predicate. The proof follows very closely the one of Theorem \ref{theorem: ErdosKakutaniK}: we need only to check that if there is a $\Delta^1_2$-strong well-ordering of $\RR$ then the sets provided by Erd\H{o}s and Kakutani's argument are uniformly $\Sigma^1_2$. For the opposite implication we apply Proposition \ref{theorem: schur}.2, the $\Sigma^1_2$ counterpart of Proposition \ref{prop: KomjathTotikK}.

As a first step we considered a Hamel basis for $\RR$. In \cite{Miller} Miller proved that if $V=L$ then there is a $\Pi^1_1$ Hamel basis. For our goal it is sufficient to provide a $\Delta^1_2$ one, under the condition $\RR \subseteq L$. The proof is straightforward and it should be well-known. However, since we have not found any reference for this proof, we show the argument.

\begin{lemma}\label{lemma: HamelBasis}
If $\RR \subseteq L$ then there exists a $\Delta^1_2$ Hamel basis for $\RR$. 
\end{lemma}
\begin{proof}
Let $\prec$ be a $\Delta^1_2$-strong well-ordering of $\RR$. Define
\[
	h \in H \iff \forall s \in (h_\prec)^{<\omega} \ \forall t \in (\QQ \setminus \bp{0})^{<\omega}( \dom(s)=\dom(t) \implies h \neq \sum_{i=0}^{\dom(s)} t(i)s(i)).
\]

By definition $H$ is rationally independent. 
We prove that $H$ generates $\RR$. Take $x \in \RR$ and prove that there exist $n \in \NN$, $h_0, \dots, h_{n-1} \in H$ and $q_0, \dots, q_{n-1} \in \QQ$ such that $x = \sum_{i=0}^{n-1} q_ih_i$. Proceed by induction on $\prec$. If $x$ has no $\prec$-predecessors then it belongs to $H$. Assume that the assertion holds for any $y$ such that  $y\prec x$. If $x \in H$ we are done. Otherwise there exist some $n\in \NN$, $x_0 \prec \dots \prec x_{n-1} \prec x$ and $q_0, \dots, q_{n-1} \in \QQ\setminus\bp{0}$ such that $x = \sum_{i=0}^{n-1} q_ix_i$. As all $x_i$ are generated by $H$, so is $x$.
\end{proof}

\begin{theorem}\label{theorem: EKdefinable}
	$\RR \subseteq L$ if and only if  there is a countable decomposition of $\RR \setminus \bp{0}$ into $\Sigma^1_2$-uniformly definable subsets consisting only of rationally independent numbers.
\end{theorem}

\begin{proof}
		\vgt{$\Rightarrow$}. Assume that $\RR \subseteq L$, and let $H$ be a $\Delta^1_2$ Hamel basis for $\RR$, which exists thanks to Lemma \ref{lemma: HamelBasis}. For any natural number $n$, for any sequence $s \in (\QQ\setminus\bp{0})^{n}$ and for any $h \in H$ define $s\cdot H^n$ and $s\cdot(H^{n-1}h)$ as in the proof of Theorem \ref{theorem: ErdosKakutaniK}. Then
		\[
					\RR \setminus \{0\} = \bigsqcup \bp{s \cdot H^{n-1}: n \in \omega, s \in (\QQ\setminus\bp{0})^{n}}.
		\]
		As observed in Remark \ref{rmk: disjointHamel} this is a disjoint union. Fix a natural number $n$ and a sequence $s\in (\QQ\setminus\bp{0})^{n}$. We want to define countably many disjoint subsets of $s \cdot H^n$ such that each one contains at most one element of $s \cdot (H^{n-1}h)$ for any $h \in H$. To this end, given an increasing finite sequence of natural numbers  $t \in \NN^{n-1}$, define $S_{s,t}$ to be the subset of $s\cdot H^n$ which consists of elements of the form $x = \sum_{i=0}^{n-1} s(i)h_i$ for some $h_0 \prec \dots \prec h_{n-1} \in H$ such that for any $i \in n-1$, $h_i$ is the $t(i)$-th predecessor of $h_{n-1}$. Therefore
		 \begin{multline*}
			   S_{s,t}=\{x : \exists h \in H \exists v \in \RR^{\leq \omega}(\IS^*(h) = v \wedge  \forall i \in n-1 (v(t(i)) \in H)\\
			   \wedge x = \sum_{i=0}^{n-2} s(i) v(t(i)) + s(n-1)h )\}.
	     \end{multline*}
		By construction $\last(x_1)\neq \last(x_2)$ for any $x_1, \ x_2 \in S_{s,t}$ such that $x_1 \neq x_2$.
		Moreover we have
	    $
			   s\cdot H^n = \bigsqcup \bp{ S_{s,t}: {t \in \NN^{n-1}}}.
	    $
	    The sets $S_{s,t}$, where $s \in \QQ^{<\omega}$ and $t \in \NN^{<\omega}$, are uniformly definable by the following $\Sigma^1_2$  formula.
		\begin{multline*}
		\psi(x,n,s, t) \iff s \in \QQ^{<\omega} \wedge t \in \NN^{<\omega} \wedge \dom(s)= n \wedge \dom(t)=n-1 \\
		\wedge \forall i \in n-2 (t(i) < t(i+1)) \wedge \exists h \exists v \in \RR^{\leq \omega}\Big[ h \in H \wedge \IS^*(h)=v\\
		\wedge \forall i \in n-1 (v(t(i)) \in H) \wedge x= \sum_{i=0}^{n-2} s(i) v(t(i))+s(n-1)h\Big].
		\end{multline*}
		
	Notice to conclude that rationally independence of the set $S_{s,t}$ holds with the same argument of Theorem \ref{theorem: ErdosKakutaniK}.

		\vgt{$\Leftarrow$}. Let $\RR \setminus \bp{0}=\bigsqcup\bp{ S_i: {i \in \omega}}$, where $S_i$ are uniformly $\Sigma^1_2$ definable rationally independent sets.
	   Let define $g: \RR \rightarrow \omega$ such that $g(0)=0$ and for any $x \in \RR \setminus \bp{0}$
	   \[
	   		g(x)=i+1 \iff x \in S_i.
	   \]
	   Since by hypothesis the $S_i$ are uniformly definable by a $\Sigma^1_2$ formula, $g$ is $\Sigma^1_2$. Suppose by contradiction that $\RR \nsubseteq L$, then by applying Proposition \ref{theorem: schur}.2 there exist $x_{00}, x_{01}, x_{10}, x_{11} \in \RR \setminus \bp{0}$ such that $x_{00}, x_{01}, x_{10}, x_{11} \in S_i$ for some $i \in \omega$ and
	   $x_{00} + x_{11} =x_{01} +x_{10}.$
	   So there are four distinct elements of $S_i$ which are rationally dependent. Contradiction.
\end{proof}

\section{Algebraically independent sets} \label{section: Zoli}

Zoli in \cite{Zoli} proved a connection between CH and the existence of a decomposition of  transcendental reals in algebraically independent sets. We provide a generalization of Zoli's result.

Given two fields $K_1 \subseteq K_2$, we say that $x \in K_2$ is algebraic over $K_1$ if there exists a polynomial $p$ in $K_1[X]$ (not null) such that $p(x)=0$. If $x \in K_2$ is not algebraic over $K_1$, then it is called transcendental over $K_1$. We denote by $\alg_{K_2}{K_1}$ the subfield of $K_2$ consisting of algebraic elements over $K_1$.
Given $x \in K_2$, $K_1(x)$ is the field extension generated by $x$.
$S \subseteq K_2$ is algebraically dependent over $K_1$ if there exist $x_0, \dots, x_n \in S$ such that $x_{n}$ is algebraic over $K_1(x_0,\dots, x_{n-1})$.  A transcendence basis $T$ is a subset of reals which is algebraically independent over $\QQ$ and maximal.
\begin{lemma}[Folklore]\label{lemma: tran2}
Let $K_1 \subseteq K_2$ be a field extension.
\begin{itemize}
\item  Let $S \subseteq K_2$. If $x \in \alg_{K_2}K_1(S) \setminus \alg_{K_2}K_1$, then $S$ is algebraically independent over $K_1(x)$.
\item  Let $T$ be a transcendence basis for $K_2$ over $K_1$. To each $x \in K_2$ there corresponds a unique minimal (finite) subset $S$ of $T$ such that $x \in \alg_{K_2} K_1(S)$.
\end{itemize}
\end{lemma}
For a proof we refer to \cite{Zoli}.

\begin{theorem}\label{theorem: ZoliK}
We have $2^{\aleph_0}\leq \kappa^+$ if and only if the set of all transcendental reals is union of $\kappa$-many algebraically independent sets.
\end{theorem}
\begin{proof}
	\vgt{$\Rightarrow$}. Assume that $2^{\aleph_0} \leq \kappa^+$, and let $T$ be a transcendence basis. Fix $f: \RR \hookrightarrow \kappa^+$ and define a well ordering of length $\kappa^+$ of reals: $x \vartriangleleft y$ if and only if $f(x) < f(y)$.  By Lemma \ref{lemma: tran2} each $x \in \RR \setminus \CQQ$ corresponds to a unique $n(x) \in \omega$ and a sequence $t_0(x) < \dots < t_{n(x)}(x) \in T$ such that 
	
	\begin{equation}\label{form: trans1a}
		x \in \alg_\RR\QQ\big(t_0(x),\dots, t_{n(x)}(x)\big) \setminus \bigcup_{i \in n(x)+1} \alg_\RR\QQ(T\setminus \bp{t_i(x)}).
	\end{equation}
	
	For any $n\in\omega$, define the set $\overline{T}_n$ of all transcendental numbers for which the cardinality of the minimum subset provided by Lemma \ref{lemma: tran2} is $n+1$; i.e. 
	\[
		\overline{T}_n=\bp{x \in \RR\setminus \CQQ : n(x) = n}.
	\]
	Thus  $\RR\setminus \CQQ = \bigcup\bp{\overline{T}_n: n \in \omega}$.  For any $n \in \NN$, $t \in \RR$ define $\overline{T}_{n, t}$ as:
	\begin{align*}
		\overline{T}_{n, t} = \{ x : x \in \overline{T}_n &\wedge  t_n(x) = t \}.
	\end{align*}
	Notice that $T$ has cardinality $\kappa^+$ and therefore $\overline{T}_{n, t}$ has cardinality at most $\kappa$. Hence fix an enumeration $\overline{T}_{n,t}=\bp{x^{n,t}_\alpha: \alpha < \kappa}$ and define
	    \begin{align*}
	    	S_{\alpha} &= \{ x_\alpha^{n,t} : n \in \omega \wedge t \in T \},
	    \end{align*}

	Observe that $\RR\setminus \CQQ$ is covered by all sets $S_\alpha$ for $\alpha \in \kappa$. In order to complete this proof we have to show that any set $ S_{\alpha}$ is algebraically independent.  Fix $\alpha \in \kappa$ in order to prove that any $x_0, \dots, x_{k-1} \in S_\alpha$ are algebraically independent. We prove it by induction over $k$. 
	
	Assume that $k=0$, then since $S_\alpha \subseteq \RR \setminus \CQQ$ we are done. Now assume that it holds for $k$ and prove it for $k+1$. Therefore consider $x_0, \dots, x_{k} \in S_\alpha$ . By construction, for any two elements of $S_\alpha$ we have $t_n(x)\neq t_n(y)$. Thus without loss of generality we can assume that $t_n(x_k) > t_n(x_i)$ for any $i\in k$. Assume by contradiction that there exists $i \in k+1$ such that $x_i \in \alg_\RR\QQ(x_0, \dots, x_{i-1},x_{i+1},\dots, x_k)$. There are two cases.
	\begin{itemize}
	\item If $i=k$. Therefore 
	\begin{multline*}
		x_k \in \alg_\RR\QQ(x_0,\dots, x_{k-1}) \subseteq \alg_\RR\QQ(\bp{t_i(x_j): i \in n, j \in k}) \\
		\subseteq \alg_\RR\QQ(T \setminus \bp{t_n(x_k)}).
	\end{multline*}
	This is a contradiction with (\ref{form: trans1a}).
	\item If $i\neq k$. By inductive hypothesis the set $\bp{x_0, \dots, x_{k-1}}$ is algebraically independent, therefore it must exist
	\[
		q \in \QQ[X_0,\dots, X_{k}]\setminus\QQ[X_0,\dots X_{i-1},X_{i+1},\dots, X_{k-1}]
	\]
	such that $q(x_0,\dots, x_{k})=0$. But this yields that $x_k \in \alg_\RR\QQ(x_0,\dots, x_{k-1})$ and this is impossible as proved in the previous case.
	\end{itemize} 
	
	\vgt{$\Leftarrow$}. Apply Theorem \ref{theorem: ErdosKakutaniK}, since any algebraically independent subset is rationally independent. 
\end{proof}

Note that the algebraically independent sets provided in the proofs from $2^{\aleph_0} \geq \kappa^+$ are disjoint. In Theorem \ref{theorem: ZoliK}, as in Zoli's original argument, we proved that the transcendental reals are a disjoint union of $\kappa$-many algebraically independent sets. However, since any algebraically independent set is contained in some transcendence basis, we obtain the following.

\begin{corollary}\label{cor:Zoli}
$2^{\aleph_0}\leq \kappa^+$ if and only if the set of all transcendental reals is union of $\kappa$-many transcendence bases for $\RR$.
\end{corollary}

\subsection{Definable counterpart}

The first step we need to show that if all reals are constructible, then there is a $\Delta^1_2$ transcendence basis.

\begin{lemma}\label{lem:tran1}
Assume that $\RR \subseteq L$. 
\begin{enumerate}
\item For any $y_0, \dots, y_n \in \RR$, $\alg_\RR\QQ(y_0,\dots y_n)$ is $\Sigma^1_1$.
\item There exists a $\Delta^1_2$ transcendence basis.
\end{enumerate}
\end{lemma}

\begin{proof}
Let $\prec$ be a $\Delta^1_2$-strong well-ordering of $\RR$.
\begin{enumerate}
\item By definition $x \in \alg_\RR\QQ(y_0,\dots, y_n)$ if and only if
\begin{multline*}
	\exists (x_0, \dots, x_m) \in \RR^{{<}\omega} ( x_0, \dots, x_m \in \QQ(y_0,\dots, y_n) \wedge x_0 + x_1 x + \dots + x_{m}x^{m}=0).
\end{multline*}
The assertion follows by induction over $n$, since $x \in \QQ(y_0, \dots, y_n)$ if and only if
\[
		\exists (a_0,\dots,a_h,b_0,\dots, b_k) \in {\QQ(y_0,\dots, y_{n-1})}^{{<}\omega}(x= \frac{a_0 +\dots + a_h y_{n}^h }{b_0 +\dots + b_k y_{n}^k}).
\]
\item Define
\[
	x \in T \iff x \in \RR \setminus \alg_\RR(\QQ) \wedge \forall (x_0, \dots, x_n) \in (x_\prec)^{<\omega}(x \notin \alg_\RR\QQ(x_0, \dots, x_n)).
\]
We claim that $T$ is a transcendence basis for $\RR$. $T$ is algebraically independent by definition. Moreover $\RR$ is algebraic over $\QQ(T)$. Indeed we show that for any $x\in \RR \setminus \alg_\RR(\QQ)$ there exist $x_0,\dots,x_n \in T$ such that $ x \in \alg_\RR\QQ(x_0, \dots, x_n)$. By induction on $\prec$. Let $x_0 = \min_{\prec}(\RR \setminus \alg_\RR(\QQ))$, then by definition $x_0 \in T$. Assume that $x \in \RR \setminus \alg_\RR(\QQ)$ and that for any $y \prec x$ the assertion holds. We have two possibilities: either $x \in T$ or there exists some $n\in \NN$ such that $\exists x_0, \dots, x_n \prec x (x \in \alg_\RR\QQ(x_0, \dots, x_n) )$. Both in the first case and whether in the second one $x_0, \dots, x_n \in T$ we have the assertion. Then assume that we are in the second case and $x_0, \dots, x_n \notin T$. As all $x_i$ are algebraic over $T$, then also $x$ is.\qedhere 
\end{enumerate}
\end{proof}

\begin{theorem}\label{theorem: ZoliDefinable}
	$\RR \subseteq L$ if and only if the set of all transcendental reals is the disjoint union of countably many algebraically independent sets uniformly definable by a $\Sigma^1_2$ predicate.
\end{theorem}
\begin{proof}
	\vgt{$\Rightarrow$}. Assume that $\RR \subseteq L$, then by Lemma \ref{lem:tran1} there exists $T$ which is a $\Delta^1_2$ transcendence basis. For any $x \in \RR \setminus \CQQ$ and any $n \in \omega$ define $n(x)$, $t_i(x)$ and $\overline{T}_n$ as in the proof of Theorem \ref{theorem: ZoliK}. Fix a natural number $n$. We define countably many disjoint subsets of $\overline{T}_n$ which contain at most one element of $\overline{T}_{n,t}$ for any $t\in T$. To this end fix $t_0, \dots, t_{n-1} \prec t$. For any $x \in \alg_\RR\QQ(t_0, \dots, t_{n-1},t)$ there exists a polynomial $p \in \QQ[X_0, \dots, X_{n},X]$ such that $p(t_0, \dots, t_{n-1},t,x)=0$. Thus for any $p \in \QQ[X_0, \dots, X_n, X]$, $t \in \RR$ and for any $v \in \RR^n$ define
	\begin{multline*}
		\overline{T}_{n, t, p, v} = \{ x : x \in S_n \wedge \forall i \in n (v(i) = t_i(x)) \wedge t = t_n(x) \wedge p(t_0(x), \dots, t_n(x),x)=0 \}.
	\end{multline*}
	Observe that $\overline{T}_{n, t, p, v}$ is uniformly $\Delta^1_2$ definable by the following formula. 
		\begin{multline*}
		\varphi(x,t,n,p,v) = x \notin \CQQ \wedge  \dom(v)= n \wedge  x \notin \alg_\RR\QQ(\bp{v(j): j\in n})\\
		\wedge \forall i \in n ( x \notin \alg_\RR\QQ(\bp{v(j): j\in n, j\neq i}\cup\bp{t})) \wedge  p(v(0), \dots, v(n-1), t ,x)= 0.
		\end{multline*}
		Up to now the sets $\overline{T}_{n, t, p, v}$ are not disjoint, since any $x$ belongs to $\overline{T}_{n, t, p, v}$ for several polynomials $p$. However since $p \in \QQ[X_0, \dots, X_n, X]$, it can be coded with a natural number by a $\Delta^1_1$ map $m: \QQ^{<\omega}\to \NN$. Therefore we can define a $\Delta^1_2$ formula which provides a partition of $\overline{T}_n$, by choosing the polynomial with the minimal code:
		\begin{multline*}
			\varphi^*(x,t,n,p,v) = \varphi(x,t,n,p,v) \wedge \forall q \in \QQ^{<\omega}(m(q)<m(p) \implies \neg \varphi(x,t,n,q,v)).
		\end{multline*}
		Put  $\overline{T}^*_{n,t,p,v} =\bp{x : \varphi^*(x,t,n,p,v)}$. Since $p(t_0, \dots, t_{n-1}, t, X) \in \QQ[X]$, it has finitely many roots. Given a root $x$, let $l$ be the number roots which are smaller than $x$ with respect to $\prec$. Hence there is a 1-1 correspondence between $\overline{T}_{n,t}$ and the set of tuples $(t_0, \dots, t_{n-1},p, l)$ for some $t_0, \dots, t_{n-1} \prec t$, $p \in \QQ[X_0, \dots, X_n, X]$ and $l \leq \deg(p(v(0), \dots, v(n-1), t, X))$. 
		
		For any $n, l \in\NN$, $p \in \QQ[X_0, \dots, X_n, X]$ and for any increasing finite sequence of natural numbers $s \in \NN^{<\omega}$ define
	    \begin{multline*}
	    	S_{n, p,s, l} = \{ x : \exists w \in \RR^{\leq \omega}\exists v \in \RR^n ( \IS^*(t_n(x))=w \wedge \forall i \in n (v(i)=w(s(i))) \\
	    	\wedge x \in \overline{T}^*_{n,t_n(x),p,v} \wedge |\bp{y : y\prec x} \cap \overline{T}^*_{n,t_n(x),p,v}| = l)\},
	    \end{multline*}
	    By construction $x,y \in S_{n, p, s, l}$ and $t_n(x)=t_n(y)$ yield $x=y$. Moreover $S_{n, p,s, l}$ can be uniformly defined by the following $\Sigma^1_2$ formula.
	    \begin{multline*}
	    \psi(x, n,p,s,l) \iff \exists t \in \RR \exists w\in \RR^{\leq \omega}\exists v \in \RR^{<\omega}( \IS^*(t)= w   \wedge  \forall i \in n(w(s(i)) = v(i))\\ 
	    \wedge \varphi(x,t,n,p,v)  \wedge \exists u \in \RR^{\leq \omega}[\IS^*(x)=u \wedge \dom(u) \geq l 
	    \wedge (\exists s' \in \NN^{l}\\
	    (\forall i \in l-1 (s'(i) < s'(i+1)) \wedge \forall i\in l \varphi(u(s'(i)),t,n,p,v)) \wedge (\dom(u) > l \\
	     \implies \forall s' \in \NN^{l+1}(\exists i \in l (s'(i) \geq s'(i+1)) 
	    \vee \exists i\in {l+1} \neg\varphi(u(s'(i)),t,n,p,v))))].
	    \end{multline*}
	We have 
	\[
	    \RR\setminus \CQQ = \bigsqcup\bp{ S_{n,p,s,l} : n, l \in \NN, s \in \NN^{<\omega}, p \in \QQ[X_0, \dots, X_n, X] }.
	\]
	In order to complete this proof we have to show that any set $ S_{n,p,s,l}$ is algebraically independent. Since, by construction for any two elements $x,y\in S_{n,p,s,l}$, $t_n(x) \neq t_n(y)$, the argument is exactly the one shown in Theorem \ref{theorem: ZoliK}.
	
	\vgt{$\Leftarrow$}. Since any algebraically independent subset is rationally independent, the assertion follows by Theorem \ref{theorem: EKdefinable}, 
\end{proof}

In his work Zoli proved that CH holds if and only if the set of all transcendental reals is the disjoint union of countably many algebraically independent sets $S_i$. Corollary \ref{cor:Zoli} follows since if $S_i$ is algebraically independent then we can define a transcendence basis $T_i$ which contains $S_i$ as follows: 
\[x \in T_i \iff  x \in S_i \vee ( x \notin \alg_\RR \QQ(S_i) \wedge \forall y_0, \dots,y_n \prec x (x \notin \alg_\RR \QQ(S_i \cup \bp{y_0, \dots, y_n}))).
\]
However this basis is $\Pi^1_2$ and up to now we did not find a $\Sigma^1_2$ formula which defines it. Therefore our definable version of Zoli's result deal with algebraically independent sets and not with transcendence bases.

\section{Polynomial avoidance} \label{section: Schmerl}

We present some results by Schmerl \cite{Schmerl} about polynomial avoidance and an equivalence by Erd\H{o}s and Komj\'{a}th \cite{ErdosKomjath} in order to obtain the correspondent $\Sigma^1_2$ definable counterparts.

We say that a polynomial 
$
	p\in \RR[X_0,\dots, X_{k-1}]
$
is a \emph{$(k,n)$-ary polynomial} if every $X_i$ is a $n$-tuple of variables. 
Given a $(k,n)$-ary polynomial $p(x_0,\dots, x_{k-1})$, a coloring $\chi: \RR^n \to \omega$ \emph{avoids} it if for any $r_0, \dots, r_{k-1}\in \RR^n$ distinct and monochromatic with respect to $\chi$, $p(r_0,\dots, r_{k-1})\neq 0$. Moreover the polynomial $p(x_0,\dots, x_{k-1})$ is \emph{avoidable} if there exists a coloring which avoids it.

\begin{definition}[Schmerl]
Let $m \in \omega$ and $k\in \omega \setminus\bp{0,1}$.
\begin{itemize}
\item A function $\alpha: A_0\times A_1\times\dots\times A_{m-1} \to B_0\times B_1\times\dots\times B_{m-1}$ is \emph{coordinately induced} if  for every $i\in m$ there is a function $\alpha_i: A_i \to B_i$ such that
\[
	\alpha(a_0,\dots, a_{m-1})= (\alpha_0(a_0),\dots, \alpha_{m-1}(a_{m-1})).
\]
\item A function $g:A^m \to B$ is \emph{one-one in each coordinate} if whenever $a_0,\dots, a_{m-1},b \in A$ and $b \neq a_i$ for some $i\in m$, then 
\[
	g(a_0, \dots a_{i-1}, a_i, a_{i+1},\dots, a_{m-1})\neq g(a_0, \dots, a_{i-1}, b , a_{i+1}, \dots a_{m-1}).
\]
\item Assume that $p(x_0, \dots, x_{k-1})$ is a $(k,n)$-ary polynomial. For each $m\in \omega$ we say that  $p(x_0, \dots, x_{k-1})$ is \emph{$m$-avoidable} if for each definable function $g:\RR^m \to \RR^n$ which is one-one in each coordinate and for distinct $e_0, \dots, e_{k-1}\in (0,1)^m$ there is a coordinately induced $\alpha:\RR^m \to \RR^m$ such that
$
	p(g\alpha(e_0), \dots, g\alpha(e_{k-1}))\neq 0.
$
\end{itemize}
\end{definition}

Instead of $\RR^m$ we can use also $(0,1)^m$, $(a,b)^m$ or any open $m$-box, since there is a $\Delta^1_1$ bijection between them.

\begin{theorem}[Schmerl]\label{theorem: Schmerl1}\label{theorem: Schmerl2} ~
\begin{enumerate}
\item If CH does not hold then every avoidable polynomial is $2$-avoidable.
\item If CH holds then every $1$-avoidable polynomial is avoidable.
\end{enumerate}
\end{theorem}

\subsection{Schmerl's equivalences}

In fact the statements studied by Schmerl are equivalences respectively with $\neg$CH and CH. To show it we need to recall Erd\H{o}s-Komj\'{a}th's equivalence. Erd\H{o}s and Komj\'{a}th proved that CH holds if and only if the plane can be colored with countably many colors with no monochromatic right-angled triangle, where a right-angled triangle is monochromatic if its vertices are. 

\begin{notation}\label{notation}
Let $\tilde{p}(x_0, x_1, x_2)$ be the following $(3,2)$-ary polynomial:
\[
	\tilde{p}(x_0, x_1, x_2) = \lVert x_1-x_0\rVert^2 +  \lVert x_2-x_0\rVert^2  -  \lVert x_1-x_2\rVert^2.
\]
\end{notation}

Observe that given distinct $a_0, a_1, a_2 \in \RR^2$,  $\tilde{p}(a_0,a_1,a_2) = 0$ if and only if $a_0$, $a_1$ and $a_2$ form a right-angled triangle. Hence

\begin{theorem}[Erd\H{o}s, Komj\'{a}th]
CH holds if and only if  
$\tilde{p}(x_0,x_1,x_2)$ is avoidable.
\end{theorem}

In \cite{Schmerl}, Schmerl also proved that $\tilde{p}(x_0, x_1, x_2)$ is $1$-avoidable and it is not $2$-avoidable. Since this result is crucial to prove our goal, let us recall the proof.

\begin{lemma}[Schmerl]\label{lemma: avoidptilde}
The $(3,2)$-ary polynomial $\tilde{p}(x_0,x_1,x_2)$ is $1$-avoidable and it is not $2$-avoidable.
\end{lemma}
\begin{proof}
First of all we prove that it is $1$-avoidable. Indeed given $g: \RR \to \RR^2$ and $e_0 \neq e_1 \neq e_2\in \RR$ define $\alpha:\RR \to \RR$ as follows:
\[
	\alpha(x)= y \iff (x=e_2 \wedge y=e_1)\ \vee \ (x \neq e_2 \wedge y=x).
\]
Hence $\alpha(e_0)=e_0$, $\alpha(e_1)=e_1$ and $\alpha(e_2)=e_1$, therefore $g(\alpha(e_0))\neq g(\alpha(e_1)) = g(\alpha(e_2))$ since $g$ is one-one in each coordinate. This yields
$
	p(g(\alpha(e_0)), g(\alpha(e_1)), g(\alpha(e_2)))\neq 0.
$

To prove that it is not $2$-avoidable let $g:\RR^2 \to \RR^2$ be the identity function and put $e_0 = (0,0)$, $e_1=(1,0)$ and $e_2= (0,1)$. They form a right-angled triangle. Let $\alpha: \RR^2 \to \RR^2$ be any coordinately induced function, hence $\alpha(e_0)$, $\alpha(e_1)$ and $\alpha(e_2)$ form a right-angled triangle, eventually degenerate. So 
\[
	\tilde{p}(g(\alpha(e_0)), g(\alpha(e_1)), g(\alpha(e_2)))=\tilde{p}(\alpha(e_0), \alpha(e_1), \alpha(e_2))= 0. \qedhere
\] 
\end{proof}

\begin{proposition}\label{proposition: viceversaSchmerl}~
\begin{enumerate}
\item If any avoidable polynomial is $2$-avoidable then $\neg$CH holds. 
\item If any $1$-avoidable polynomial is avoidable then CH holds. 
\end{enumerate}
\end{proposition}
\begin{proof}~
\begin{enumerate}
\item Assume that CH holds, then by Erd\H{o}s and Komj\'{a}th's equivalence $\tilde{p}(x_0,x_1,x_2)$ is avoidable. Then, by hypothesis it is $2$-avoidable. Contradiction by Lemma \ref{lemma: avoidptilde}.
\item By Lemma \ref{lemma: avoidptilde}, $\tilde{p}(x_0,x_1,x_2)$ is $1$-avoidable. Then, by hypothesis, it is avoidable. Therefore, again by Erd\H{o}s-Komj\'{a}th equivalence, CH holds. \qedhere
\end{enumerate}
\end{proof}

\subsection{Auxiliary results}\label{subsection: aux}

Recall that our goal is to provide $\Sigma^1_2$ definable counterparts for both the results by Schmerl and Erd\H{o}s and Komj\'{a}th's equivalence. To this end we need some technical facts. The first one is the $\Sigma^1_2$ version of a lemma used to prove Theorem \ref{theorem: Schmerl2}.2. 

\begin{lemma}\label{lemma: G}
Assume that $\RR \subseteq L$, then there is a $\Sigma^1_2$ function $G:\RR^{{<}\omega} \to \omega$ such that whenever $a,\ b \in \RR^{{<}\omega}$
\begin{itemize}
\item if $G(a)=G(b)$ then $|a|=|b|$;
\item if $G(a)=G(b)$ and $\max(a)=\max(b)$ then $a=b$.
\end{itemize}
\end{lemma}
\begin{proof}
Let $i: \omega^{<\omega} \to \omega$ be any $\Sigma^1_2$ injection. For each $x \in \RR$, let $F_x: x \to \omega$ be the $\Sigma^1_2$ injection defined by: 
\[
	F_x(y) = n \iff \exists v \in \RR^{\leq\omega} (\IS^*(x)=v \wedge v(n)=y).
\]
By definition $F_x(y) = n$ holds if $y$ is the $n$-th predecessor of $x$. Then consider $a \in \RR^{{<}\omega}$. If $a = \emptyset$ put $G(a)=0$ otherwise let $x=\max a$,  $G(a)=i(F_{x}[a \setminus \bp{x}])$. 
\end{proof}

The second fact we need is the following lemma, for the proof we refer to \cite{SchmerlCountable}. 

\begin{lemma}[Schmerl]\label{lemma: beta}
Let $T$ be a transcendence basis for $\RR$ over $\CQQ$. Let $l$ be a natural number, $q_i \in \QQ$ for any $i\in 2l$,  $D = (q_0,q_1)\times \dots\times(q_{2l-2},q_{2l-1})$ and let
$
	h: D^k \to \RR
$ 
be an $\CQQ$-definable analytical function. Suppose that for any $j \in k$, $\overline{t}_j=(t_{0,j}, \dots, t_{l-1,j}) \in T^l \cap D$ and $h(\overline{t}_0,\dots, \overline{t}_{k-1})=0$. If $\beta:\bp{t_{i,j}: i \in l, j \in k}\to \RR$ is such that $\beta''\overline{t}_j=(\beta(t_{0,j}), \dots, \beta(t_{l-1,j}))\in D$ for any $j \in k$, then $h(\beta''\overline{t}_0 ,\dots, \beta''\overline{t}_{k-1})=0$.
\end{lemma}

The coloring which witnesses the $\Sigma^1_2$-avoidance of the given polynomial in the original proof of Theorem \ref{theorem: Schmerl2}.2 is defined by using the Implicit Function Theorem. Therefore, to prove the effective version, we also need a basic fact about uniformly continuous functions. Recall that a function $f$ is uniformly continuous if 
\[
	\forall \varepsilon >0 \exists \delta >0 \forall x, y \in \dom(f) (|x-y |< \delta \implies |f(x)-f(y)|<\varepsilon).
\]
\begin{definition}\label{def:Quc}
Let $a$, $b$ be rational numbers and let $l$ be a natural number. Given a function $f : (a,b)^l\cap\QQ^l \to \RR$ define: 
\begin{multline*}
\psi(f, (a,b)^l) \iff \forall \varepsilon \in \QQ^+ \exists \delta \in \QQ^+ \forall q_1, q_2 \in (a,b)^l \cap \QQ^l \\
(|q_1 -q_2| < \delta \implies |f(q_1)-f(q_2)| < \varepsilon).
\end{multline*}
We say that $f : (a,b)^l\cap\QQ^l \to \RR$ is \emph{$\QQ$-uniformly continuous} if $\psi(f,(a,b)^l)$ holds.
\end{definition}

\begin{remark}\label{Remark: UnifDiff}
If $f:(a,b)^l\to\RR$ is uniformly continuous then $f\res (a,b)^l \cap \QQ^l$ is $\QQ$-uniformly continuous.
\end{remark}

\begin{lemma}\label{Lemma: UnifDiff}
Let $a, b \in \QQ$, $l$ be a natural number and let $f: (a,b)^l \cap \QQ^l \to \RR$ be $\QQ$-uniformly continuous.
\begin{itemize}
\item There exists a $\Delta^1_1$ uniformly continuous function $f^*:(a,b)^l \to \RR$  which extends $f$.
\item Moreover if $p(x,y)$ is a polynomial such that $\forall q \in (a,b)^l\cap \QQ^l (p(q,f(q))=0)$, then
\[
	\forall r \in (a,b)^l  \ (p(r,f^*(r))=0 ).
\]
\end{itemize}
\end{lemma}
\begin{proof}
 Define $f^*:(a,b)^l \to \RR$ as follows. Let $r \in (a,b)^l$ and let 
\[
	\bp{q_n\in (a,b)^l \cap \QQ^l: n \in \omega}
\] 
be such that $\lim_{n \to \infty}q_n = r.$
We claim that $\bp{f(q_n): n \in \omega}$ is a Cauchy's sequence. Let $\varepsilon > 0$, we want to prove that there exists $N$ such that for any $n, m > N (|f(q_n)-f(q_m)|<\varepsilon)$. By hypothesis we have 
\[
	\exists \delta \in \QQ^+ (\forall q_1, q_2\in (a,b)^l \cap \QQ^l  (|q_1 - q_2| < \delta \implies |f(q_1)-f(q_2)|<\varepsilon)).
\]
Since $\lim_{n \to \infty}q_n = r$, there exists $N \in \NN$ such that for any $n, m > N (|q_n-q_m|<\delta)$. Therefore $|f(q_n)-f(q_m)|< \varepsilon$ and so $\lim_{n \to \infty}f(q_n) \in \RR$. Put $f^*(r)= \lim_{n \to \infty}f(q_n).$
Observe that $f^*(r)=l$ is $\Sigma^1_1$ (and so $\Delta^1_1$), indeed:
\[
	f^*(r)=l \iff \exists q_n \in ((a,b)^l \cap \QQ^l)^\omega (\lim_{n\to\infty}(q_n)=r \implies \lim_{n\to \infty}f(q_n)= l),
\]
where
\[
	\lim_{n\to\infty}(x_n)=y \iff \forall \varepsilon \in \QQ^+ \exists n\in \NN \forall m > n(|q_m-l|<\varepsilon).  
\]

For the second part, let $r \in (a,b)^l$ and let $\bp{q_n \in (a,b)^l \cap \QQ^l: n\in \omega}$ be such that $\lim_{n \to \infty} q_n = r$.
Then 
\[
 	p(r, f^*(r))=p(\lim_{n\to \infty}q_n, \lim_{n\to \infty}f(q_n)) = \lim_{n\to \infty}p(q_n, f(q_n))=0. \qedhere
\]
\end{proof}

\subsection{Definable counterparts}

To prove the corresponding $\Sigma^1_2$ equivalences with $\RR \subseteq L$, we first need to consider the counterparts of the definitions of avoidance. The definition of $\Sigma^1_2$-avoidance directly follows by the one of avoidance, while we have to be more careful in defining the $\Sigma^1_2$ version of $m$-avoidance. 

\begin{definition}
A $(k,n)$-ary polynomial $p(x_0,\dots, x_{k-1})$ is $\Sigma^1_2$-avoidable if there exists a $\Sigma^1_2$ coloring which avoids it.
\end{definition}

\begin{definition}
A $(k,n)$-ary polynomial $p(x_0,\dots, x_{k-1})$ is $(m,\Sigma^1_2)$-avoidable if  for each $r \in \RR\cap L $, for each $\Sigma^1_2(r)$ function $g:\RR^m \to \RR^n$ which is one-one in each coordinate and for distinct $e_0, \dots, e_{k-1}\in \RR^m$ there is $r'\in \RR$ and a coordinately induced $\alpha:\RR^m \to \RR^m$ which is  $\Sigma^1_2(r')$ and such that
$
	p(g\alpha(e_0), \dots, g\alpha(e_{k-1}))\neq 0.
$
\end{definition}

Observe that we permit $g$ to be defined with a parameter in $\RR \cap L$. To justify this definition recall that in Schmerl's definition $g$ had to be definable.

\begin{theorem}\label{theorem: 2avoidance}
	If $\RR \nsubseteq L$ then every $\Sigma^1_2$-avoidable polynomial is $(2,\Sigma^1_2)$-avoidable.
\end{theorem}

\begin{proof}
Let $k, n\in \NN$. Let $p(x_0,\dots, x_{k-1})$ be a $(k,n)$-ary polynomial which is $\Sigma^1_2$-avoidable polynomial. Then there exists a $\Sigma^1_2$ coloring $\chi: \RR^n \to \omega$ which avoids it. To prove that it is also $(2,\Sigma^1_2)$-avoidable let $r \in \RR \cap L$ and consider any $\Sigma^1_2(r)$ function $g: \RR^2 \to \RR^n$ (one-one in each coordinate) and any distinct $e_0,\dots, e_{k-1}\in \RR^2$. Then $\chi \circ g: \RR^2 \to \omega$ is also $\Sigma^1_2(r)$. Since $\RR \nsubseteq L$, then by applying Proposition \ref{theorem: CD}.1, there exist $C=\bp{c_i : i\in k}$ and $D=\bp{d_i : i \in k}$ such that $\chi\circ g \res C \times D$ is monochromatic. Hence define $\alpha_0: \RR\to \RR$ and $\alpha_1: \RR\to \RR$ by:
\begin{align*}
	\alpha_0(x)=y &\iff \bigvee_{j\in k}(x = e_j(0) \ \wedge \ y= c_0)\ \vee \ (\bigwedge_{j\in k}(x \neq e_j(0) \ \wedge \ y=x)).\\
	\alpha_1(x)=y &\iff  \bigvee_{j\in k}(x = e_j(1) \ \wedge \ y= d_j)\ \vee \ (\bigwedge_{j\in k}(x \neq e_j(1) \ \wedge \ y=x)).
\end{align*}
Then consider $\alpha: \RR^2 \to \RR^2$ such that $\alpha((x,y))= (\alpha_0(x), \alpha_1(y))$. Since $\alpha(e_j)= (c_0, d_j) \in C \times D$, $\bp{\chi\circ g\circ\alpha(e_i): i\in k}$ is monochromatic. Moreover for any $i\neq j \in k$, $g(\alpha(e_i)) \neq g(\alpha(e_j))$ since $g$ is one-one in each coordinate. Then 
$
p(g(\alpha(e_0)),\dots,g(\alpha(e_{k-1})))\neq 0
$,
since $\chi$ avoids $p$.
\end{proof}

The proof of the $\Sigma^1_2$ version of Theorem \ref{theorem: Schmerl2}.2 requires a more elaborated argument, which uses all auxiliary facts listed in Subsection \ref{subsection: aux}.

\begin{theorem}\label{theorem: 1avoidance}
	If $\RR \subseteq L$ then every $(1,\Sigma^1_2)$-avoidable polynomial is $\Sigma^1_2$-avoidable.
\end{theorem}
\begin{proof}
Let $T$ be a $\Delta^1_2$ transcendence basis of $\RR$ over $\CQQ$ provided by Lemma \ref{lem:tran1}. For any natural number $l$ and any $\overline{q} \in \QQ^{2l}$ define $\dom_{\overline{q}}= (q_0,q_1) \times \dots \times (q_{2l-2},q_{2l-1})$. 

Fix $n\in\NN$. For any $a=\ap{a_0 \dots, a_{n-1}}\in \RR^n$, $d \in \omega$, $\overline{q} \in \QQ^{2l}$ and $f_i \in \RR^{\dom_{\overline{q}} \cap \QQ^{l}}$\footnote{Note that $f_i$ can be coded as an element in $\RR^\omega$.} define $\chi(a,\ap{\dom_{\overline{q}}, f_0, \dots, f_{n-1}, d})$ if and only if there exist $(t_0, \dots, t_{l-1})\in \dom_ {\overline{q}}$ and $p_0, \dots, p_{n-1} \in \CQQ[X_0, \dots, X_{l-1},Y]$ such that 
\begin{enumerate}
\item $t_0, \dots, t_{l-1} \in T$;
\item $\forall j \in s \exists i \in n(a_i \notin \CQQ(T \setminus \bp{t_j}))$;
\item $d= G(\bp{t_0, \dots, t_{l-1}})$,  where $G$ is provided by Lemma \ref{lemma: G};
\item for any $i \in n$ $p_i(t_0, \dots, t_{l-1}, a_i)=0$; 
\item for any $i \in n$ $\psi(f_i, \dom_{\overline{q}})$, i.e. $f_i$ is $\QQ$-uniformly continuous (Definition \ref{def:Quc});
\item for any $i \in n$, $\forall \overline{q}' \in \dom_{\overline{q}}\cap \QQ^{l} p_i(\overline{q}', f_i(\overline{q}'))$.
\item for any $i \in n$, $f_i$ is one-one in each coordinate $j$ such that $a_i \notin \CQQ(T \setminus \bp{t_j})$.
\end{enumerate}

By unfolding definition, $\chi$ is $\Sigma^1_2$. For any $a=\ap{a_0 \dots, a_{n-1}}\in \RR^n$, by using Lemma \ref{lemma: tran2} there exist $t_0 < \dots < t_{l-1}\in T$ such that $a_0,\dots, a_{n-1}\in \CQQ(t_0, \dots, t_{l-1})$.  For any $i \in n$, let $p_i \in  \CQQ[X_0, \dots, X_{l-1},Y]$ such that $p_i(t_0, \dots, t_{l-1},a_i)=0$ and $\frac{\partial p_i}{\partial y}(t_0,\dots,t_{l-1},a_i)\neq 0$.
Then thanks to the Implicit Function Theorem there exist $U \subseteq \RR^l$ and $V \subseteq \RR$ and $\tilde{f}_i: U \to V$ continuous differentiable such that $(t_0,\dots, t_{l-1})\in U$ and $\tilde{f}_i(b_0,\dots b_{l-1})=c $ if and only if $p_i(b_0,\dots, b_{l-1},c)=0$. In particular $\tilde{f}_i(t_0,\dots, t_{l-1})=a_i$. 
Let $\overline{q} \in \QQ^{2l}$ be increasing and such that both $(t_0, \dots, t_{l-1})\in \dom_{\overline{q}} \subseteq U$ and $\tilde{f}_i$ is one-one in each coordinate $j$ such that $a_i \notin \CQQ(T \setminus \bp{t_j})$. For any $i \in n$, define $f_i = \tilde{f_i}\res(\dom_{\overline{q}} \cap \QQ^{l})$. 
Hence $\chi(a, \ap{\dom_{\overline{q}},f_0, \dots, f_{n-1}, G(\bp{t_0, \dots, t_{l-1}})})$. Indeed the functions provided by the Implicit Function Theorem are continuously differentiable and therefore uniformly continuous. By Remark \ref{Remark: UnifDiff}, $f_i$ is $\QQ$-uniformly continuous.

By $\Sigma^1_2$ Novikov Kondo Addison Uniformization Theorem (see e.g. \cite[4E4]{Moschovakis}) there exists a $\Sigma^1_2$ function $\chi^*$ which uniformizes $\chi$. $\chi^*$ has a countable range, since there are countably many polynomials and given $\dom_{\overline{q}}$ and $p_0, \dots, p_{n-1}$ there exists a unique tuple of functions which satisfies the conditions above.

\begin{claim}
The coloring $\chi^*$ avoids any $(1,\Sigma^1_2)$-avoidable polynomial.
\end{claim}

Let $p \in \CQQ[X_0,\dots, X_{k-1}]$ be a $(k,n)$-ary polynomial which is not avoided by $\chi^*$. We want to prove that this polynomial is not $(1,\Sigma^1_2)$-avoidable. Let $a^0, \dots, a^{k-1}\in \RR^n$ be distinct, monochromatic in color $\ap{\dom, f_0, \dots, f_{n-1}, d}$, and such that $p(a^0,\dots, a^{k-1})=0.$

For each $j\in k$ let $a^j \in\CQQ(t_{0,j},\dots, t_{l-1,j})$, where $t_{i,j}\in (q_{2i},q_{2i+1})$. For any $i\in n$, let $f_i^*$ be the witness of Lemma \ref{Lemma: UnifDiff} for $f_i$ and put $f = (f^*_0,\dots, f^*_{n-1})$. Define $g:(q_{2l-2},q_{2l-1}) \to \RR^n$ such that 
\[
	g(x)= f(t_{0,0},\dots, t_{l-2,0},x).
\]
Note that $g$ is $\Sigma^1_1$ with parameters in $L$ since $\RR \subseteq L$. Moreover by (2) and (7) $g$ is injective. For each $j\in k$ put $e_j=t_{l-1,j}\in (q_{l-1},r_{l-1})$. By Lemma \ref{lemma: G} they are all distinct. Indeed if there are $j, j' \in k$ such that $e_j =e_{j'}$, since $G(\bp{t_{0,j},\dots, t_{l-1,j}})= d = G(\bp{t_{0,j'},\dots, t_{l-1,j'}})$, we have for any $i \in l$, $t_{i,j}=t_{i,j'}$. Therefore for any $m \in n$: $a^j_m = f_m(t_{0,j}, \dots, t_{l-1,j}) = f_m(t_{0,j'}, \dots, t_{l-1,j'}) = a^{j'}_m.$
 To obtain our assertion we need to show that for any $\alpha: (q_{2l-2},q_{2l-1}) \to (q_{2l-2},q_{2l-1})$,
$
	p(g(\alpha(e_0)), \dots, g(\alpha(e_{k-1})))=0.
$

Fix a function $\alpha$. Let $h:((q_0, q_1)\times \dots \times (q_{2l-2},q_{2l-1}))^k \to \RR$ be such that 
\begin{multline*}
	h((y_{0,0}\dots, y_{l-1,0}),\dots, (y_{0,k-1},\dots, y_{l-1,k-1})) \\
	= p(f(y_{0,0},\dots, y_{l-1,0}),\dots,f(y_{0,k-1},\dots, y_{l-1,k-1})).
\end{multline*}
And for any $i \in n-1$ put $\beta(t_{0,i},\dots, t_{l-1,i})=(t_{0,0},t_{1,0},\dots, t_{l-2,0},\alpha(t_{l-1,i}))$. Then by Lemma \ref{lemma: beta}\footnote{For any $i \in n$, if $f_i(q'_0, \dots, q'_{l-1})=x$ then $p(q'_0, \dots, q'_{l-1},x)=0$. It follows $f_i \in \CQQ^{\dom_{\overline{q}} \cap \QQ^{l}}$. Hence both $f$ and $h$ are $\CQQ$ definable analytical functions.},
$
	p(f\beta(t_{0,0},\dots, t_{l-1,0}),\dots,f\beta(t_{0,k-1},\dots, t_{l-1,k-1}))=0
$
and
\begin{multline*}
	p(g\alpha(e_0),\dots, g\alpha(e_{k-1}) = p(g(\alpha(t_{l-1,0})), \dots, g(\alpha(t_{l-1,k-1}))) \\
	=p(f\beta(t_{0,0},\dots, t_{l-1,0}),\dots,f\beta(t_{0,k-1},\dots, t_{l-1,k-1}))=0.  
\end{multline*}
\qedhere
\end{proof}

By using Theorem \ref{theorem: 1avoidance} we can prove the $\Sigma^1_2$ version of the equivalence by Erd\H{o}s and Komj\'{a}th. In order to do that observe that the $(3,2)$-ary polynomial  $\tilde{p}(x_0,x_1,x_2)$ we defined in Notation \ref{notation} is $(1, \Sigma^1_2)$-avoidable and not $(2,\Sigma^1_2)$-avoidable.The proof directly follows the one of Lemma \ref{lemma: avoidptilde}.
\begin{theorem}
	$\RR \subseteq L$ if and only if there exists a $\Sigma^1_2$ coloring of the plane with countably many colors with no monochromatic right-angled triangle. 
\end{theorem}
\begin{proof}
\vgt{$\Rightarrow$}. Assume that $\RR \nsubseteq L$. By applying Proposition \ref{theorem: CD}.1 for every $f: \RR^2 \to \omega$  $\Sigma^1_2$ there exists a monochromatic rectangle and so we are done. 

\vgt{$\Leftarrow$}. Assume that $\RR \subseteq L$. Recall that $\tilde{p} (x_0,x_1,x_2)$ is $(1, \Sigma^1_2)$-avoidable. By Theorem \ref{theorem: 1avoidance} it is $\Sigma^1_2$-avoidable. 
\end{proof}

To conclude observe that by using the $\Sigma^1_2$-version of Erd\H{o}s-Komj\'{a}th equivalence and the properties of $\tilde{p}(x_0,x_1,x_2)$ we can easily prove the vice versa of Theorem \ref{theorem: 2avoidance} and of Theorem \ref{theorem: 1avoidance} as in Proposition \ref{proposition: viceversaSchmerl}, therefore we get

\begin{proposition}
The following are equivalent:
\begin{itemize}
\item $\RR \subseteq L$;
\item every $\Sigma^1_2$-avoidable polynomial is $(2,\Sigma^1_2)$-avoidable;
\item every $(1,\Sigma^1_2)$-avoidable polynomial is $\Sigma^1_2$-avoidable.
\end{itemize}
\end{proposition}

\section{Open Questions}

The arguments presented for Theorem \ref{theorem: EKdefinable} and Theorem \ref{theorem: ZoliDefinable} require that the countably many subsets are uniformly definable. This is needed to provide a $\Sigma^1_2$ coloring in the proofs of the second implications. Therefore the first natural questions are
\begin{question}
 Assume that the set of all real numbers can be decomposed into a countably many (possibly non uniformly) $\Sigma^1_2$ definable rationally independent subsets. Does $\RR \subseteq L$ hold? 
\end{question}
\begin{question}
 Assume that the set of all transcendental reals is the union of countably many (possibly non uniformly) $\Sigma^1_2$ definable algebraically independent subsets. Does $\RR \subseteq L$ hold? 
\end{question}

As observed in Section \ref{section: Zoli} our definable version of Zoli's equivalence produces countably many algebraically independent subsets which are uniformly $\Sigma^1_2$ definable. Therefore we wonder whether $\RR \subseteq L$ implies that the set of all transcendental reals is the union of countably many uniformly $\Sigma^1_2$ definable transcendence bases. In particular
\begin{question}
 Given an algebraically independent subset $A \subseteq \RR$ which is $\Delta^1_2$ is it possible to define a $\Delta^1_2$ transcendence basis which contains $A$? 
\end{question}

A more general natural question which arises from this work is  
\begin{question}\label{question1}
For which inner model does the $\Sigma^1_3$ (or the more general $\Sigma^1_n$) definable version hold?
\end{question}
By considering the arguments used in the proofs, Question \ref{question1} can be reformulated as: \vgt{which inner model has the perfect set property for $\Sigma^1_n$, a $\Delta^1_n$ strong well-ordering and $\Sigma^1_n$ absoluteness?} As suggested by Alessandro Andretta a possible model for $\Sigma^1_{2n}$ could be the inner model for $n$-many Woodin cardinals. Anyway, since as far as we know there is not a proof of Mansfield Theorem's analogous for such models, the argument is not straightforward. 

Finally, it seems that for any equivalent form of CH is possible to prove its $\Sigma^1_2$ definable version. Hence we wonder: 
\begin{question}
Is there some general argument which provides, by assuming large cardinal hypotheses, the existence of a proof for the $\Sigma^1_2$ corresponding counterpart?
\end{question} 

\paragraph{Acknowledgement.}
I would like to thank Alessandro Andretta for his precious suggestions, remarks and for introducing me to the topic in the first place. I am grateful to Rapha\"{e}l Carroy for his careful readings, advices and corrections. I am also thankful to the Logic Group of Torino for the patience and the useful comments during my talks. 

\bibliographystyle{plain}
\bibliography{biblio}
\end{document}